\documentclass[12pt]{article}
\title{An identity in the Bethe subalgebra of $\CC[\Sn]$}

\makeatletter\def\@fnsymbol#1{\ensuremath{\ifcase#1\or \dagger\or \ddagger\or\mathsection\or \mathparagraph\or \|\or **\or \dagger\dagger \or \ddagger\ddagger \else\@ctrerr\fi}}\makeatother

\author{Kevin Purbhoo\thanks{Research supported by NSERC Discovery Grant RGPIN-04741-2018.}}

\usepackage{latexsym, amsmath, amscd, hyperref, bbm}
\usepackage[shortlabels]{enumitem}
\usepackage{amsthm}
\usepackage{graphicx, xspace, ifthen, rotating, pict2e}
\usepackage[svgnames]{xcolor}
\usepackage[charter]{mathdesign}
\usepackage{algpseudocode, algorithmicx}
\usepackage{tikz}

\setlist{topsep=.5ex, itemsep=.5ex, parsep=0ex, partopsep=0ex}

\allowdisplaybreaks[1]
\newcommand{\CC}{\mathbb{C}}

\newcommand{\QQ}{\mathbb{Q}}
\newcommand{\RR}{\mathbb{R}}

\newcommand{\PP}{\mathbb{P}}

\newcommand{\calD}{\mathcal{D}}
\newcommand{\calF}{\mathcal{F}}

\newcommand{\calY}{\mathcal{Y}}

\newcommand{\gl}{\mathfrak{gl}}

\newcommand{\Wr}{\mathrm{Wr}}
\newcommand{\Gr}{\mathrm{Gr}}

\newcommand{\pker}{\mathop{\mathrm{pker}}}

\newcommand{\sgn}{\mathrm{sgn}}
\newcommand{\ord}{\mathrm{ord}}

\newtheorem{lemma}{Lemma}[section]
\newtheorem{theorem}[lemma]{Theorem}

\newtheorem{proposition}[lemma]{Proposition}
\newtheorem{corollary}[lemma]{Corollary}

\newtheorem{conjecture}[lemma]{Conjecture}

\theoremstyle{definition}
\newtheorem{example}[lemma]{Example}

\newtheorem*{ack}{Acknowledgements}
\newtheorem{remark}[lemma]{Remark}

\numberwithin{equation}{section}
\numberwithin{figure}{section}
\numberwithin{table}{section}

\definecolor{DarkBlue}{rgb}{0, 0.1, 0.55}
\definecolor{DarkRed}{rgb}{0.45, 0, 0}

\newcommand{\defn}[1]{\textbf{\textit{#1}}}

\newcommand{\conjugate}[1]{#1^*}
\newcommand{\scell}{\mathcal{X}^\lambda}
\newcommand{\dualscell}{\mathcal{X}^{\conjugate{\lambda}}}

\newcommand{\du}{\partial_u}

\newcommand{\End}{\mathrm{End}}

\newcommand{\indicial}{\mathrm{Ind}}

\newcommand{\Sn}{\mathfrak{S}_n}
\newcommand{\symgp}{\mathfrak{S}}
\newcommand{\idSn}{1_{\Sn}}

\newcommand{\bethe}{\mathcal{B}_n}
\newcommand{\Dalg}{\mathbb{D}}
\newcommand{\SP}{\mathrm{SP}}
\newcommand{\Dcoeff}[2]{\langle #1 \rangle_{#2}}
\newcommand{\SgnRep}{\mathbb{A}}
\newcommand{\schur}{\mathsf{s}}
\newcommand{\uea}{U\big(\gl_m(\CC[t])\big)}
\newcommand{\exterior}[1]{\wedge^{#1}}
\newcommand{\extalg}{\exterior{\bullet}}

\begin{document}
\maketitle

\begin{abstract}
As part of the proof of the Bethe ansatz conjecture for the Gaudin
model for $\gl_n$,
Mukhin, Tarasov, and Varchenko described a 
correspondence between inverse Wronskians of polynomials 
and eigenspaces of the Gaudin Hamiltonians.
Notably, this correspondence afforded the first proof of the 
Shapiro--Shapiro conjecture.
In the present paper, we give an identity in the group algebra of 
the symmetric group,
which allows one to establish the correspondence directly, 
without using the Bethe ansatz.
\end{abstract}

%
%  2020 MSC:
%  20C30  Representations of finite symmetric groups
%  81R05  Finite-dimensional groups and algebras motivated by physics and their representations
%  05E10  Combinatorial aspects of representation theory
%
%

%%%%%%%%%%%%%%%%%%%%%%%%%%%%%%%%%%%%%%%%%%%%%%%%%%%%%%%%%
%
%
\section{Introduction}

Let $f_1(u), \dots, f_m(u) \in \CC(u)$ be linearly independent 
rational functions.  The \defn{Wronskian}
\[
   \Wr(f_1, \dots, f_m) = 
    \begin{vmatrix}
    f_1 & f_1' & f_1'' & \dots & f_1^{(m-1)} \\
    f_2 & f_2' & f_2'' & \dots & f_2^{(m-1)} \\
    \vdots & \vdots & \vdots &  & \vdots \\
    f_m & f_m' & f_m'' & \dots & f_m^{(m-1)} \\
    \end{vmatrix}
\,,
\]
is also a rational function, which, up to a scalar multiple, 
depends only on the span of $f_1, \dots, f_m$.
It is therefore reasonable to talk about the Wronskian of
a finite dimensional subspace $V \subset \CC(u)$:
if $V  = \langle f_1, \dots, f_m \rangle$ is the subspace
of $\CC(u)$ spanned by $f_1,\dots, f_m$, we define
$\Wr_V(u) \in \CC(u)$ to the unique scalar multiple of
$\Wr(f_1, \dots, f_m)$ which is a \defn{monic rational function}, 
i.e. a ratio of two monic polynomials.
We will mainly be interested in the case where 
the basis elements $f_1, \dots, f_m$ are polynomials, in
which case $\Wr_V$ is a monic polynomial.

Given a polynomial $w(u)  = (u+z_1) \dotsm (u+z_n) \in \CC[u]$, 
the \defn{inverse Wronskian problem} for $w(u)$ is to 
find all subspaces of polynomials $V \subset \CC[u]$ such that
$\Wr_V = w$.
There are finitely many such $V$ of any particular dimension $m$.
Moreover, if one can find all the $n$-dimensional solutions, then it
is straightforward to find all solutions of any other dimension
(see Proposition~\ref{prop:dimensions});
we will therefore focus on the case $m=n$.  

The inverse Wronskian problem appears
in many guises throughout mathematics. It can be reformulated 
as a Schubert intersection problem, or in terms of linear series on
$\PP^1$, or in terms of rational curves in with 
with prescribed flexes.  It is also a special case 
of the pole placement problem in control theory \cite{EG-pole}.  
The survey \cite{Sot-F} discusses many of these alternate formulations 
along with a variety applications.

There is also a deep connection with representation theory
and quantum integrable systems.
Over a series of papers (see \cite{MTV-reality}),
Mukhin, Tarasov and Varchenko showed that the problem of finding these
solutions is equivalent to the problem of finding eigenvectors of the
Bethe algebra for the Gaudin model.  The Bethe algebra is defined as
a commutative subalgebra of the universal enveloping algebra
$\uea$ \cite{MTV-transverse}; 
however by Schur--Weyl duality,
it has a quotient $\bethe(z_1, \dots, z_n)$ which can be identified 
with a commutative subalgebra of $\CC[\Sn]$, 
the group algebra of the symmetric group~\cite{MTV-Sn}.
$\bethe(z_1, \dots, z_n)$ is called the Bethe subalgebra 
of $\CC[\Sn]$ (of Gaudin type).

Briefly, here's how the equivalence works.
One concretely writes down certain operators (the \emph{Gaudin
Hamiltonians}), which in this paper are
denoted $\beta_{k,l}^- \in \CC[\Sn]$, $k,l \leq n$.   (The ``$-$'' in our
notation requires some explanation; this will be provided shortly.)
These operators commute pairwise, and they are generators 
of $\bethe(z_1, \dots, z_n)$.
We combine them to form a linear differential operator with coefficients
in $\CC[\Sn] \otimes \CC(u)$:
\[
  \calD^-_n
          = \frac{1}{w(u)} \bigg(
       \sum_{k,l} (-1)^{k} \beta^-_{k,l} u^{n-k-l} \du^{n-k}\bigg)
\,.
\]
One can then restrict this differential operator to any eigenspace 
$E$ of $\bethe(z_1, \dots, z_n)$, which gives a 
scalar valued differential operator $\calD^-_{E}$ of order $n$, with 
coefficients in $\CC(u)$.
\begin{theorem}[Mukhin--Tarasov--Varchenko]
\label{thm:invwrsols}
The kernel of $\calD^-_{E}$ is an $n$-dimensional vector space 
$V_E \subset \CC[u]$, which is a solution of the inverse Wronskian problem 
for $w(u)$.  
Furthermore all $n$-dimensional solutions to the inverse Wronskian 
problem are of this form.
\end{theorem}

Theorem~\ref{thm:invwrsols} is far from obvious.
Arguably the most mysterious 
part is the dimension of the space of polynomials in the kernel.
In general, if one writes down a linear differential equation of 
order $n$ with coefficients in $\CC(u)$, 
it is rare for it to have any rational solutions,
let alone an $n$-dimensional space of
polynomial solutions.
Of course, one can write down equations for
when this occurs, but these are difficult to work with explicitly,
and checking directly that the operators $\beta^{-}_{k,l}$ satisfy
these equations seems to be impractical.
Mukhin, Tarasov and Varchenko's proof of Theorem~\ref{thm:invwrsols}
is part of a larger body of work on the Bethe ansatz, 
a technique from mathematical
physics for finding the eigenvectors to certain problems involving
commuting operators.  In a nutshell, they show that when one applies
the Bethe ansatz method to the Gaudin model, the Bethe ansatz equations 
for finding the eigenvectors can be reinterpreted as equations for solving
the inverse Wronskian problem.
The formulation in terms of
$\bethe(z_1, \dots, z_n) \subset \CC[\Sn]$ is derived from theorems
about the infinite dimensional Bethe algebra inside 
$\uea$ using Schur--Weyl duality.

The main goal of this paper is to give an account of 
Theorem~\ref{thm:invwrsols}, which is short, mostly self-contained, 
operates strictly inside $\CC[\Sn]$, 
and does not involve 
finding the eigenvectors of the Bethe algebra.  
Our main result (Theorem~\ref{thm:main}) is an identity in 
$\bethe(z_1, \dots, z_n)$, which accomplishes this.
We introduce a second operator $\calD^+_n$, 
which is related to $\calD^-_n$ by an anti-involution of the algebra
of $\CC[\Sn]$-valued linear differential operators.
All minus signs in the formula are changed to plusses, and the order 
of factors is reversed from left to right.  
\[
  \calD^+_n
          = 
   \bigg(\sum_{k,l} \du^{n-k}u^{n-k-l} \beta^+_{k,l} \bigg) \frac{1}{w(u)}
\,.
\]
The coefficients $\beta^+_{k,l}$ are given by a similar formula 
to $\beta^-_{k,l}$, but again, without signs.
We show that the elements $\beta^+_{k,l}$ are also generators 
for $\bethe(z_1,\dots,z_n)$. 
This means one can also restrict $\calD^+_n$ to any eigenspace $E$
of $\bethe(z_1, \dots, z_n)$ to get a scalar valued differential operator
$\calD^+_E$.
\begin{theorem}
\label{thm:main}
In $\Dalg[\Sn]$, the algebra of $\CC[\Sn]$-valued linear differential
operators, we have the identity
\[
\calD^+_n \calD^-_n = \du^{2n}
\,.
\]
\end{theorem}
We can now argue as follows. 
If $E$ is any eigenspace of $\bethe(z_1,\dots, z_n)$, we obtain
the scalar valued differential operator identity
\[
     \calD^+_E \calD^-_E = \du^{2n}
\,.
\]
Since $\ker(\du^{2n})$ is a $2n$-dimensional subspace of $\CC[u]$,
and $\ker(\calD^+_E), \ker(\calD^-_E)$ both have dimension at most $n$,
we see that $V_E = \ker(\calD^-_{E})$ must be an
$n$-dimensional subspace of $\ker(\du^{2n})$; in particular $V_E$
is an $n$-dimensional space of polynomials.
It now follows readily (see Corollary~\ref{cor:invwrsols}) 
that $V_E$
is a solution to the inverse Wronskian problem for $w(u)$.  
The fact that every solution arises in this way follows as well,
because we know how many solutions there are 
to each of the two problems (see Remark~\ref{rmk:bijectivecorrespondence}).

An important consequence of Theorem~\ref{thm:invwrsols} is the
reality theorem, conjectured by B. and M. Shapiro in the mid-1990s
and proved by Mukhin, Tarasov and Varchenko in \cite{MTV-reality}
(see also~\cite{EG-elementary,LP,MTV-transverse, Sot-F}).
If $z_1, \dots, z_n$ are real, then the operators $\beta^-_{k,l}$
are real and self-adjoint with respect to the standard inner product 
on $\CC[\Sn]$ (for which the group elements form an orthonormal basis); 
hence $\bethe(z_1, \dots, z_n)$ is diagonalizable over $\RR$, 
and the entire argument above goes through with $\RR$ in place of $\CC$.

\begin{theorem}
If $z_1, \dots, z_n \in \RR$, then all solutions to the inverse Wronskian
problem for $w(u)$ are real.
\end{theorem}

A natural question is whether there are analogous results for $\calD^+_n$.
For a partial answer, consider the inverse Wronskian problem 
for \emph{rational functions}:
given $g(u) \in \CC(u)$, find $V \subset \CC(u)$ such that
$\Wr_V = g$.  Theorem~\ref{thm:main} 
implies that if $E$ is an eigenspace of $\bethe(z_1,\dots,z_n)$, 
then $\ker(\calD^+_{E})$ is an $n$-dimensional
subspace of $\CC(u)$, which is a solution to
the inverse Wronskian problem for the rational function $\frac{1}{w(u)}$.
However, in this case we are not getting \emph{all} rational solutions: 
unlike the polynomial inverse Wronskian problem, the rational inverse 
Wronskian has infinitely many solutions of any given dimension.
We discuss this further in Section~\ref{sec:conclusion}.

This paper is structured as follows.  Sections~\ref{sec:diff} 
and~\ref{sec:bethe} provide background on 
the fundamental differential operator of a subspace $V \subset \CC(u)$, 
and on the Bethe subalgebra of $\CC[\Sn]$.  
The proof of Theorem~\ref{thm:main} is given in Section~\ref{sec:main}.  
Sections~\ref{sec:commute}, \ref{sec:schubert} and \ref{sec:duality}
establish additional properties of the algebra $\bethe(z_1, \dots, z_n)$, 
beginning with a combinatorial proof of commutativity, and
culminating in the
fact that the operators $\beta^+_{k,l} \in \CC[\Sn]$ are 
generators (Theorem~\ref{thm:commute}).  
We conclude with a discussion of the mysterious operator $\calD^+_n$, 
and other open questions in Section~\ref{sec:conclusion}.
In keeping with our stated objectives, our exposition includes 
proofs of 
known results whenever the original proof was based on the 
Bethe ansatz or derived from identities in algebras other than 
$\CC[\Sn]$, e.g. using Schur--Weyl duality.

\begin{ack}
This work became possible thanks to discussions
during the Fields Institute Thematic Program on Combinatorial Algebraic
Geometry in 2016, with Joel Kamnitzer, Frank Sottile, and David Speyer.
I thank Vitaly Tarasov and an anonymous referee, for pointing out 
some recent relevant references.
Innumerable calculations for this project were carried out using 
\texttt{Sage} \cite{sage}.
\end{ack}

%%%%%%%%%%%%%%%%%%%%%%%%%%%%%%%%%%%%%%%%%%%%%%%%%%%%%%%%%
%
%
\section{Fundamental differential operators}
\label{sec:diff}

Let $\Dalg = \CC(u)[\du]$ denote the algebra of complex valued
linear differential operators
in variable $u$, with rational function coefficients.
The algebra $\CC(u)$ of rational functions is a commutative subalgebra of 
$\Dalg$, and $\Dalg$ has the commutation relations
\begin{equation}
\label{eqn:noncommute}
   \du g - g \du = g'
\end{equation}
for $g = g(u) \in \CC(u)$.
Every element 
$\Psi \in \Dalg$ can be expressed uniquely in the form
\[
   \Psi = \sum_{j=0}^m  \psi_j(u) \du^{j}
\]
where $\psi_0(u), \dots, \psi_m(u) \in \CC(u)$.
If $\psi_m(u) \neq 0$, then $m = \ord(\Psi)$ is called the 
\defn{order} of $\Psi$, 
and we say $\Psi$ is a \defn{monic operator} if $\psi_m(u) = 1$.
Write $\Dcoeff{\Psi}{j} = \psi_j(u)$, to mean the coefficient of $\du^j$
in this canonical representation.

We view $\Psi$ as a linear differential operator 
$\Psi : \CC(u) \to \CC(u)$, via
\[
   \Psi : g(u) \mapsto \Dcoeff{\Psi g}{0} = \sum_{j=0}^m  \psi_j(u) g^{(j)}(u)
\,.
\]
Write $\ker(\Psi ) \subset \CC(u)$ for the kernel of this operator,
and $\pker(\Psi ) = \ker(\Psi ) \cap \CC[u]$ for the subspace of polynomials
in $\ker(\Psi )$.  Note that when we write $\Psi  g$ or $\Psi g(u)$, 
this will \emph{always} mean the
product of $\Psi $ and $g$ in $\Dalg$, and should not be confused with
the rational function $\Dcoeff{\Psi g}{0}$ obtained by applying the differential
operator $\Psi$ to $g$.  

From the general theory of 
linear ordinary differential equations, we have the following basic
inequalities (see e.g. \cite[\S 3.32]{Ince}).

\begin{proposition}
\label{prop:order}
For any non-zero $\Psi  \in \Dalg$, 
\[
    \dim \pker(\Psi ) \leq \dim \ker(\Psi ) \leq \ord(\Psi )
\,.
\]
\end{proposition}

Let $V \subset \CC(u)$ be a finite dimensional $\CC$-linear subspace 
of $\CC(u)$.   Choose any basis $(f_1, \dots, f_m)$ for $V$.
The \defn{fundamental differential operator} of $V$ 
is the monic operator $D_V \in \Dalg$, defined by the determinantal formula
\[
  D_V = 
    \frac{1}{\Wr(f_1, \dots, f_m)}
    \begin{vmatrix}
    f_1(u) & f_1'(u) & f_1''(u) & \dots & f_1^{(m)}(u) \\
    f_2(u) & f_2'(u) & f_2''(u) & \dots & f_2^{(m)}(u) \\
    \vdots & \vdots & \vdots &  & \vdots \\
    f_m(u) & f_m'(u) & f_m''(u) & \dots & f_m^{(m)}(u) \\
    1 & \du & \du^2  & \dots & \du^m \\
    \end{vmatrix}
\,.
\]
This definition is independent of the choice of basis.
Here, we use the convention that the determinant of a $k \times k$ matrix 
$A$ with non-commuting entries is defined to be the ``row-expansion''
\[
  |A| = \sum_{\sigma \in \symgp_k} \sgn(\sigma) 
  A_{1,\sigma(1)} A_{2,\sigma(2)} \dots A_{k, \sigma(k)}
\,.
\]
Equivalently, viewing as $D_V$ a differential operator $\CC(u) \to \CC(u)$, 
we have
\[
   \Dcoeff{D_V g}{0} = \frac{\Wr(f_1, \dots, f_m, g)}{\Wr(f_1, \dots, f_m)}
\,.
\]
The numerator is zero if and only if $f_1, \dots, f_m, g$ 
are linearly dependent, 
i.e.  if and only if $g \in V$.  Hence we see that $\ker(D_V) = V$.

Not every monic operator in $\Dalg$ is a fundamental differential 
operator.  We have the following elementary characterization.

\begin{proposition}
\label{prop:FDO}
Suppose $\Psi  \in \Dalg$ is a monic operator of order $m$.
\begin{enumerate}[(i)]
\item
$\Psi  = D_V$ for some finite dimensional $V \subset \CC(u)$ if and only if 
$\dim \ker(\Psi ) = m$.
\item
$\Psi  = D_V$ for some finite dimensional $V \subset \CC[u]$ if and only if 
$\dim \pker(\Psi ) = m$.
\end{enumerate}
In either case, if $\Psi  = D_V$, then $\Dcoeff{\Psi}{m-1} 
= - \frac{\Wr_V'}{\Wr_V}$.
\end{proposition}

\begin{proof}
If $\Psi = D_V$ then $\dim \ker(\Psi) = \ord(\Psi) = m$.  
Conversely, if $\ker(\Psi) = V$,
and $\dim V = m$, then $\Psi$ and $D_V$ are both monic differential
operators of order $m$, with kernel $V$.  Therefore $V \subseteq 
\ker (D_V -\Psi)$, so $\dim \ker (D_V -\Psi) > \ord (D_V -\Psi)$; by
Proposition~\ref{prop:order}, this 
is only possible if $D_V -\Psi = 0$.  
This proves (i) and a similar argument proves (ii).
The final statement is a straightforward calculation, and follows 
directly from the definitions of $D_V$ and $\Wr_V$.
\end{proof}

\begin{corollary}
\label{cor:invwrsols}
Let $g(u) \in \CC(u)$ be a monic rational function, and
let $\Psi  \in \Dalg$ be a monic operator of order $m$. 
\begin{enumerate}
\item[(i)]
If $\dim \ker(\Psi) = m$, and
$\Dcoeff{\Psi}{m-1} = - \frac{g'(u)}{g(u)}$, then $\ker(\Psi)$ 
is a solution to 
the (rational) inverse Wronskian problem for $g(u)$.
\item[(ii)] If $g(u) \in \CC[u]$, $\dim \pker(\Psi) = m$, and
$\Dcoeff{\Psi}{m-1} = - \frac{g'(u)}{g(u)}$, then $\pker(\Psi)$ 
is a solution to 
the (polynomial) inverse Wronskian problem for $g(u)$.
\end{enumerate}
\end{corollary}

We now describe the relationship between solutions of
the (polynomial) inverse Wronskian problem of different dimensions.

\begin{proposition}
\label{prop:dimensions}
Let $w(u) = (u+z_1)\dotsm (u+z_n)$ be a polynomial of degree $n$.
\begin{enumerate}
\item[(i)]
$V \subset \CC[u]$ 
is an $m$-dimensional solution to the inverse Wronskian problem
for $w(u)$ if and only if $\pker(D_V \du)$ is an $(m+1)$-dimensional 
solution.

\item[(ii)]
If $m \geq n$, then every $(m+1)$-dimensional solution is of the form
$\pker(D_{V} \du)$, for some $m$-dimensional solution $V$.
\end{enumerate}
\end{proposition}

\begin{proof}
We have $\pker(D_V \du) = \{f(u) \in \CC[u] \mid f'(u) \in V\}$, which
is $(m+1)$-dimensional, and $\Dcoeff{D_V\du}{m} = \Dcoeff{D_V}{m-1}$.
Statement therefore (i) follows from Proposition~\ref{prop:FDO}.

For (ii), suppose $V' \subset \CC[u]$ is an $(m+1)$-dimensional solution.
Let $(f_1, \dots, f_{m+1})$ be a basis for $V'$.  We may assume that
$\deg(f_1) > \dots > \deg(f_{m+1})$.  Then 
$\Wr(f_1, \dots, f_{m+1})$ has degree 
$\sum_{i=1}^{m+1} \deg(f_i) - \frac{m(m+1)}{2} \geq (m+1) \deg(f_{m+1})$.
Therefore when $m \geq n$, $f_{m+1}$ must be a constant.
It follows that
$D_{V'} = \Psi \du$ for some $\Psi$. Now
Proposition~\ref{prop:FDO} implies that $\Psi = D_V$ for some $m$-dimensional
solution $V$.
\end{proof}

Using Proposition \ref{prop:dimensions},
the inverse Wronskian problem for $w(u)$ reduces to the 
case where $\dim V = \deg(w) = n$.  If we can we find
all $n$-dimensional solutions, then we obtain all 
$m$-dimensional solutions, $m=1,2,3,\dots$, as follows:
for $m < n$, take all subspaces $V \subset \CC[u]$ 
such that $\pker (D_V \du^{n-m})$ is an $n$-dimensional solution;
for $m > n$, take all subspaces $\pker(D_V \du^{m-n}) \subset \CC[u]$ such
that $V$ is an $n$-dimensional solution.  In the remaining sections, 
our discussion of the inverse Wronskian problem will focus exclusively
on the case where $\dim V = \deg (w)$.

%%%%%%%%%%%%%%%%%%%%%%%%%%%%%%%%%%%%%%%%%%%%%%%%%%%%%%%%%
%
%
\section{The Bethe subalgebra of $\CC[\Sn]$}
\label{sec:bethe}

Let $\Sn$ denote the symmetric group of permutations of
$[n] =\{1, \dots, n\}$, and let
$\CC[\Sn]$ denote the group algebra of $\Sn$.
Write $\idSn$ for the identity element of $\Sn$.

As before, let
$w(u) = (u+z_1)\dotsm (u+z_n) \in \CC[u]$, where $z_1, \dots, z_n$ are
complex numbers.  
For a subset $X \subseteq [n]$, 
write $z_X = \prod_{i \in X} z_i$.
Let $\symgp_X = \{\sigma \in \Sn \mid \sigma(i) = i \text{ for }i \notin X\}$
be the subgroup $\Sn$ which permutes only the elements of $X$.  
Define elements $\alpha^\pm_X \in \CC[\Sn]$, as follows:
\[
   \alpha^+_X = \sum_{\sigma \in \symgp_{X}} \sigma
\qquad
\qquad
   \alpha^-_X = \sum_{\sigma \in \symgp_{X}} \sgn(\sigma) \sigma
\,.
\]
For $k,l \leq n$, let
\[
    \beta_{k,l}^\pm = \sum_{\substack{|X| = k, |Y| = l \\ X \cap Y = \emptyset}}
             \alpha_X^\pm z_Y
\,.
\]
In particular, note that
\begin{align*}
\beta^\pm_{0,n} &= z_1 z_2 \dotsm z_n \cdot \idSn 
&
\beta^\pm_{1,n-1} &= z_1 z_2 \dotsm z_n \sum_{i=1}^n \frac{1}{z_i}  \cdot \idSn 
\,.
\end{align*}

Define $\bethe^-(z_1, \dots, z_n)$ (resp. $\bethe^+(z_1, \dots, z_n)$)
to be the subalgebra of $\CC[\Sn]$ generated by the 
group algebra elements $\beta_{k,l}^-$ (resp. $\beta^+_{k,l}$), $k,l \leq n$.  
Let $\bethe(z_1, \dots, z_n)$ denote the algebra generated
by both the $\beta_{k,l}^-$ and $\beta_{k,l}^+$ operators.

\begin{theorem}
\label{thm:commute}
The elements $\beta_{k,l}^{\pm}$ commute pairwise.  
Furthermore, 
\[
 \bethe^-(z_1, \dots, z_n) = \bethe^+(z_1, \dots, z_n)
 = \bethe(z_1, \dots, z_n)
\,.
\]
\end{theorem}

The proof of Theorem~\ref{thm:commute} is given in Section~\ref{sec:commute}.

The commutative algebra $\bethe(z_1, \dots, z_n)$ is called
is called the \defn{Bethe subalgebra} of $\CC[\Sn]$
of Gaudin type.  Certain properties of this subalgebra depend on the
numbers $z_1, \dots, z_n$.  For example, the dimension of 
$\bethe(z_1, \dots, z_n)$
depends on $z_1, \dots, z_n$; in some cases $\bethe(z_1, \dots, z_n)$
is semisimple, but not always.
However, in all cases it contains $Z(\CC[\Sn])$, 
the centre of $\CC[\Sn]$.

\begin{theorem}
\label{thm:centre}
The elements $\beta^-_{0,0}, \beta^-_{1,0}, \dots, \beta^-_{n,0}$
generate $Z(\CC[\Sn])$.
\end{theorem}

If $\gamma \in \bethe(z_1, \dots, z_n)$, let $\gamma(t)$ denote the
element obtained from $\gamma$ by the substitution 
$(z_1, \dots, z_n) \mapsto (z_1+t, \dots, z_n+t)$.  For example,
we have
\[
   \beta_{k,n-k}^\pm(t) = \sum_{l=0}^{n-k} \beta_{k,l}^\pm t^{n-k-l}
\,.
\]
From this, it is not hard to see that $\bethe(z_1, \dots, z_n)$
is generated by the elements 
$\beta^\pm_{k,n-k}(t)$, $k =0, \dots, n$, $t \in \CC$.
In particular the Bethe subalgebra is \defn{translation invariant}, i.e.
$\bethe(z_1, \dots, z_n) = \bethe(z_1+t, \dots, z_n+t)$.

Consider the algebra $\Dalg[\Sn] = \Dalg \otimes \CC[\Sn]$.  
For ease of notation, we will implicitly identify
$\Psi \in \Dalg$, with $\Psi \otimes \idSn \in \Dalg[\Sn]$, 
and $\gamma \in \CC[\Sn]$ with
$1 \otimes \gamma \in \Dalg[\Sn]$.   
The elements of $\Dalg[\Sn]$ can
be uniquely written in form $\sum_{\sigma \in \Sn} \Psi_\sigma \sigma$,  
where $\Psi_\sigma \in \Dalg$.  
Note that we have 
$\Psi \gamma = \gamma \Psi$ for all $\Psi \in \Dalg$ and 
$\gamma \in \CC[\Sn]$.  

The operators $\calD^-_n, \calD^+_n \in \Dalg[\Sn]$ which appear 
in Theorem~\ref{thm:main} can now be written more concisely as:
\begin{align*}
  \calD^+_n
          &= 
       \left(\sum_{k=0}^n \du^{n-k} \beta^+_{k,n-k}(u) \right) 
   \frac{1}{w(u)}
\qquad &
  \calD^-_n
          &= \frac{1}{w(u)}
       \left(\sum_{k=0}^n (-1)^{k} \beta^-_{k,n-k}(u) \du^{n-k}\right)
\\
          &= 
      \du^n + \du^{n-1} \tfrac{w'(u)}{w(u)} + \dots + 
        \tfrac{\beta^+_{n,0}}{w(u)}
&
          &= 
         \du^n - \tfrac{w'(u)}{w(u)} \du^{n-1} + \dots + 
        (-1)^n \tfrac{\beta^-_{n,0}}{w(u)}
\,.
\end{align*}

\begin{remark}
The operators $\calD^+_n$ and $\calD^-_n$ are
related by an anti-involution of $\Dalg[\Sn]$. 
If $\omega : \Dalg[\Sn] \to \Dalg[\Sn]$ is the anti-automorphism
defined by $\omega(\du) = -\du$, $\omega(g) = g$, 
$\omega(\sigma) = \sgn(\sigma) \sigma^{-1}$, for all $g(u) \in \CC(u)$, 
$\sigma \in \Sn$, then $\omega(\calD^-_n) = (-1)^n \calD^+_n$.
\end{remark}

Let $\lambda \vdash n$ be a partition, and let $M^\lambda$ denote the 
irreducible $\CC[\Sn]$-module associated to $\lambda$.
An \defn{eigenspace of the Bethe algebra of type $\lambda$}
is a maximal linear subspace $E \subset M^\lambda$, such that each
operator $\gamma \in \bethe(z_1, \dots, z_n)$ acts as a scalar 
$\gamma_E$ on $E$.  In particular, for any eigenspace $E$ of the
Bethe algebra, we obtain scalars $\beta^\pm_{k,l,E} \in \CC$,
polynomials $\beta^\pm_{k,n-k,E}(u) 
= \sum_{l=0}^{n-k} \beta_{k,l,E}^\pm u^{n-k-l} \in \CC[u]$,
and scalar valued differential operators $\calD^-_E, \calD^+_E \in \Dalg$:
\begin{align*}
  \calD^+_E
          &= 
       \left(\sum_{k=0}^n \du^{n-k} \beta^+_{k,n-k,E}(u) \right) 
   \frac{1}{w(u)}
&
  \calD^-_E
          &= \frac{1}{w(u)}
       \left(\sum_{k=0}^n (-1)^{k} \beta^-_{k,n-k,E}(u) \du^{n-k}\right)
\,.
\end{align*}
Note that 
$\Dcoeff{\calD^-_E}{n-1} = - \frac{w'(u)}{w(u)}$, and
$\Dcoeff{\calD^+_E}{n-1} = + \frac{w'(u)}{w(u)}$.
Thus Theorem~\ref{thm:main} and Corollary~\ref{cor:invwrsols}
imply that $\ker (\calD^-_E)$ is a solution
to the inverse Wronskian problem for $w(u)$, and
$\ker (\calD^+_E)$ is a solution
to the inverse Wronskian problem for $\frac{1}{w(u)}$.

\begin{example}
\label{ex:21}
Take $w(u) = u^3-3u$, i.e. $(z_1, z_2, z_3) = (\sqrt{3}, -\sqrt{3}, 0)$.
Consider the $2$-dimensional $\CC[\symgp_3]$-module $M^{21}$, 
in which the elementary transpositions $(1\,2)$ and $(2\,3)$
are represented by the matrices
$[\begin{smallmatrix} 0 & -1 \\ -1 & 0  \end{smallmatrix}]$
and
$[\begin{smallmatrix} 1 & 0 \\ 1 & -1  \end{smallmatrix}]$, respectively%
\footnote{Following the conventions used by \texttt{Sage} \cite{sage}.}.
Then
$\beta^-_{3,0}= \alpha_{\{1,2,3\}}^-$ acts as zero, and 
\[
\beta^-_{2,1}(u) = \alpha_{\{2,3\}}^- (u+z_1) + 
\alpha_{\{1,3\}}^- (u+z_2) + 
\alpha_{\{1,2\}}^- (u+z_3)
\]
is represented by the matrix
\[
     \begin{bmatrix}
      0 & 0 \\
      -1 & 2 
    \end{bmatrix} (u+z_1)
     +
     \begin{bmatrix}
      2 & -1 \\
      0 & 0 
    \end{bmatrix} (u+z_2)
     +
     \begin{bmatrix}
      1 & 1 \\
      1 & 1 
    \end{bmatrix} (u+z_3)
     =
     \begin{bmatrix}
      3u-2\sqrt{3} & \sqrt{3} \\
      -\sqrt{3} & 3u+2\sqrt{3}
    \end{bmatrix} 
\,.
\]
The eigenspaces $\mathcal{B}_3(\sqrt{3}, -\sqrt{3}, 0)$
of type $\lambda = 21$ are therefore the eigenspaces of the
matrix $[\begin{smallmatrix} -2 & 1 \\ -1 & 2 \end{smallmatrix}]$,
which are
   $E_1 = \mathrm{span}\, 
       [\begin{smallmatrix} 1 \\ 2+\sqrt{3} \end{smallmatrix}]$
and
   $E_2 = \mathrm{span}\, 
       [\begin{smallmatrix} 2+\sqrt{3} \\ 1 \end{smallmatrix}]$.
Restricting $\calD^-_n$ to each eigenspace, we obtain
\[
   \calD^-_{E_1} = \du^3 - \tfrac{3u^2-3}{u^3-3u} \du^2 
    + \tfrac{3u+3}{u^3-3u} \du
 \qquad \qquad
   \calD^-_{E_2} = \du^3 - \tfrac{3u^2-3}{u^3-3u} \du^2 
    + \tfrac{3u-3}{u^3-3u} \du
\,.
\]
One can check that
$\ker (\calD^-_{E_1}) = \langle u^4+4u^3, u^2-2u,1 \rangle$
and $\ker (\calD^-_{E_2}) = \langle u^4-4u^3, u^2+2u,1 \rangle$,
which are indeed solutions to the inverse Wronskian problem for $u^3-3u$.
There are two more solutions, which come from the $1$-dimensional
$\CC[\Sn]$-modules $M^3$ and $M^{111}$.
\end{example}

\begin{remark}
Our exposition differs from \cite{MTV-Sn} in the following respect.
In \cite{MTV-Sn}, the Bethe subalgebra of $\CC[\Sn]$ is defined to be 
the algebra generated by the elements $\beta^-_{k,l}$, 
whereas here, we have defined it to be the algebra generated by all elements
$\beta^\pm_{k,l}$.  Theorem~\ref{thm:commute}
asserts that these definitions agree.
The fact that $\bethe^-(z_1, \dots, z_n)$ is commutative is the
content of \cite[Proposition 2.4]{MTV-Sn}, and
one can easily deduce that 
$\bethe^+(z_1, \dots, z_n)$ is also commutative.
However, the fact that $\bethe(z_1, \dots, z_n)$ is commutative does 
not seem to follow directly; we prove this in 
Section~\ref{sec:commute}.  From here we deduce 
Theorem~\ref{thm:invwrsols}, and use it to show
that all three algebras are equal.  This establishes that
$\bethe(z_1, \dots, z_n)$ is generated by the elements $\beta^-_{k,l}$,
and it is also generated by the elements $\beta^+_{k,l}$, as
asserted in the introduction.
Theorem~\ref{thm:centre} is \cite[Proposition 2.1]{MTV-Sn}, and
we include a short proof in Section~\ref{sec:schubert}.
\end{remark}

%%%%%%%%%%%%%%%%%%%%%%%%%%%%%%%%%%%%%%%%%%%%%%%%%%%%%%%%%
%
%
\section{Proof of the main identity}
\label{sec:main}

In this section, we prove Theorem~\ref{thm:main}.

For each $a \in [n]$ let $q_a(u) = \frac{1}{u+z_a}$, and for a subset
$X \subseteq [n]$, let
$q_X(u) = \prod_{a \in X} q_a(u)$.  
The operators $\calD^{\pm}_n$ can be rewritten as follows.
\begin{align*}
  \calD^+_n
          &= \sum_{X \subseteq [n]} \du^{n-|X|} q_X(u) \alpha^+_X
&
  \calD^-_n
          &= \sum_{Y \subseteq [n]} (-1)^{|Y|} \alpha^-_Y q_Y(u) \du^{n-|Y|}
\end{align*}

A \defn{supported permutation} $\sigma_Z$ is a permutation $\sigma$,
together with a set $Z \subseteq [n]$, such that $\sigma \in \symgp_Z$.
The set $Z$ is called the support of $\sigma_Z$.
Let $\SP_n$ denote the set of all supported permutations.

Given $\sigma_Z \in \SP_n$ and a subset $A \subseteq Z$,
let $\calF_{\sigma_Z,A}$ be the set of pairs of supported permutations
$(\delta_X,\varepsilon_Y)$,
such that $X \cup Y = Z$, $X \cap Y = A$, and $\delta \varepsilon = \sigma$.
Thus $\calF_{\sigma_Z,A}$ is the set of factorizations
of $\sigma$ into two supported permutations, with some conditions on the
supports.

Consider the differential operators 
$F_{\sigma_Z,A}, F_{\sigma_Z} \in \Dalg$, 
\[
    F_{\sigma_Z,A} = \sum_{(\delta_X,\varepsilon_Y) \in \calF_{\sigma_Z,A}}
                 (-1)^{|Y|}\sgn(\varepsilon) \,
                    \du^{|Y|-|A|} \, 
                    q_A(u) q_Z(u) \,
                    \du^{|X|-|A|}
\]
and
\[
    F_{\sigma_Z} = \sum_{A\subseteq Z} F_{\sigma_Z,A}
\,.
\]
When we expand the product $\calD^+_n \calD^-_n$, and reorganize the
terms, we get the following formula.

\begin{lemma}
\label{lem:expansion}
\[
 \calD^+_n \calD^-_n 
  = 
\sum_{\sigma_Z \in \SP_n}
        \du^{n-|Z|} \,
       F_{\sigma_Z} \sigma \,
       \du^{n-|Z|}
\,.
\]
\end{lemma}

\begin{proof}
We have
\begin{align*}
\calD^+_n \calD^-_n 
&= \sum_{X,Y \subseteq [n]} 
          (-1)^{|Y|} 
          \du^{n-|X|} \, \alpha^+_X \alpha^-_Y \, q_X q_Y \, \du^{n-|Y|}
\\
&= \sum_{X,Y \subseteq [n]} 
  \sum_{\delta \in \symgp_X} 
   \sum_{\varepsilon \in \symgp_Y}
          (-1)^{|Y|} \sgn(\varepsilon)
          \du^{n-|X|} \, \delta \varepsilon \, 
            q_{|X \cap Y|} q_{|X \cup Y|} \,
        \du^{n-|Y|}
\\
&= \sum_{A \subseteq Z \subseteq [n]}
   \sum_{\substack{X,Y \subseteq [n]\\X \cup Y = Z \\ X \cap Y = A}}
  \sum_{\sigma \in \symgp_Z}
  \sum_{\substack{\delta \in \symgp_X \\ \varepsilon \in \symgp_Y  
        \\ \delta \varepsilon = \sigma}}
          (-1)^{|Y|} \sgn(\varepsilon)
          \du^{n-|Z|+|Y|-|A|} \, \sigma \, q_A q_Z \, 
        \du^{n-|Z|+|X|-|A|}
\\
&= \sum_{A \subseteq Z \subseteq [n]} 
  \sum_{\sigma \in \symgp_Z}
          \du^{n-|Z|} F_{\sigma_Z,A} \sigma \du^{n-|Z|}
\\
  &= 
   \sum_{\sigma_Z \in \SP_n}
        \du^{n-|Z|} 
       \, F_{\sigma_Z} \sigma\, 
       \du^{n-|Z|}
\,.
\qedhere
\end{align*}
\end{proof}

We now show that almost all of the terms on the right hand side of
Lemma~\ref{lem:expansion} are equal to zero.

\begin{lemma}
\label{lem:bigintersection}
If $|A| \geq 2$, then $F_{\sigma_Z,A} = 0$.
In particular, 
\[
   F_{\sigma_Z} 
      = F_{\sigma_Z,\emptyset}
                    + \sum_{a=1}^n F_{\sigma_Z,\{a\}}
\,.
\]
\end{lemma}

\begin{proof}
As $|A| \geq 2$, there exists a transposition $\tau \in \symgp_A$.  Then 
$(\delta_X, \varepsilon_Y) \leftrightarrow ((\delta\tau)_X, (\tau\varepsilon)_Y)$
defines a sign reversing involution on the set $\calF_{\sigma_Z,A}$,
so $F_{\sigma_Z,A} = 0$.
\end{proof}

To analyze the cases $|A| \leq 1$,
consider the $\CC$-bilinear map $\Phi : \CC[s,t] \times \CC(u) \to \Dalg$, 
defined by $\Phi(s^it^j, g) = \du^i g(u) \du^j$, for $g(u) \in \CC(u)$,
$i,j \geq 0$.
Notice that the operator $F_{\sigma_Z,A}$ is equal to 
$\Phi(p_{\sigma_Z,A}, q_Aq_Z)$, for the polynomial
\begin{equation}
\label{eqn:polyF}
   p_{\sigma_Z,A}(s,t) 
          = \sum_{(\delta_X, \varepsilon_Y) \in \calF_{\sigma_Z,A}}
               (-1)^{|Y|}\sgn(\varepsilon) \, s^{|Y|-|A|} t^{|X|-|A|}
\,.
\end{equation}
The following identity is a reformulation
of the commutation relations~\eqref{eqn:noncommute}.
\begin{proposition} 
\label{prop:commutationidentity}
For any $p(s,t) \in \CC[s,t]$, and $g(u) \in \CC(u)$,
\[
   \Phi\big((s-t)p, g\big) = \Phi(p,g')
\,.
\]
\end{proposition}

\begin{proof}
Since $\Phi$ is bilinear, if suffices to prove this for $p = s^it^j$,
in which case we have 
\[
\Phi\big((s-t)p,g\big) = 
 \du^i (\du g(u) - g(u) \du) \du^j = \du^i g'(u) \du^j = \Phi(p,g')
\,.
\qedhere
\]
\end{proof}

\begin{lemma}
\label{lem:allzero}
If $Z$ is non-empty then $F_{\sigma_Z} = 0$.
\end{lemma}

To simplify some of the notation, we present the argument for the 
case $Z = [n]$; 
the other cases are proved by a conceptually identical argument, 
with $\symgp_Z$ in place of $\Sn$.

\begin{proof}
Suppose $\sigma \in \Sn$ has 
cycles $\gamma_1$, \dots, $\gamma_m$. Let $\nu_i$ be 
the length of cycle $\gamma_i$ and for any subset $K \subseteq [m]$,
write $\nu_K = \sum_{i \in K} \nu_i$.
We will show that the polynomials
$p_{\sigma_{[n]},A}$, $|A| \leq 1$, are related to the following polynomial:
\[
    p_\nu(s,t) = \prod_{i=1}^m (t^{\nu_i} - s^{\nu_i})
\,.
\]

First we compute $p_{\sigma_{[n]},\emptyset}$.
By definition this is a sum over $\calF_{\sigma_{[n]},\emptyset}$,
the set of all factorizations of $\sigma$ into two supported
permutations $\delta_X$ and $\varepsilon_Y$ where $(X,Y)$ is a
partition of $[n]$.  The only way to obtain such a factorization is
to partition the cycles of $\sigma$: we must have
$\varepsilon = \prod_{i \in K} \gamma_i$ and 
$\delta = \prod_{i \notin K} \gamma_i$ for some subset 
$K \subseteq [m]$.  For this factorization, we have
$(-1)^{|Y|}\sgn(\varepsilon) = (-1)^{|K|}$, $|Y| = \nu_K$ and 
$|X| = n -\nu_K$.
Plugging this information into~\eqref{eqn:polyF}, we obtain
\[
p_{\sigma_{[n]},\emptyset}
 = 
\sum_{K \subseteq [m]}
   (-1)^{|K|}
  \,
   s^{\nu_K}
  \,
   t^{n-\nu_K}
= p_\nu
\,.
\]

Next we compute $p_{\sigma_{[n]},\{a\}}$.
Without loss of generality, we may assume
that $a$ appears in the last cycle $\sigma_m$, say
$\sigma_m = (a\, b_1\, b_2 \, b_{\nu_m-1})$.  
Consider the following cycles:
$\pi_i = (a\, b_1\, \dots\, b_{i-1})$ and 
$\pi_i' = (a\, b_i\, \dots\, b_{\nu_m-1})$.  
The factorizations of $\sigma$ into
$\delta_X$ and $\varepsilon_Y$ such that $X \cup Y = [n]$ and 
$X \cap Y = \{a\}$ are of
the form
\[
  \varepsilon = \pi_i \cdot \textstyle \prod_{i \in K} \gamma_i
\qquad
\qquad
  \delta = \pi_i' \cdot \textstyle \prod_{i \notin K} \gamma_i
\,,
\]
where $K \subseteq [m-1]$ and $1 \leq i \leq \nu_m$.
For this factorization, we have
$(-1)^{|Y|}\sgn(\varepsilon) = (-1)^{|K|+1}$,
$|Y| = \nu_K + i$ and $|X| = n-\nu_K-i+1$.
Therefore, using~\eqref{eqn:polyF},
\[
p_{\sigma_{[n]},\{a\}}
 = 
\sum_{K \subseteq [m-1]} \sum_{i=1}^{\nu_m}
(-1)^{|K|+1} s^{\nu_K+i-1} t^{n-\nu_K-i}
 = 
   - \sum_{i=1}^{\nu_m} t^{\nu_m-i} s^{i-1} \cdot
\prod_{i=1}^{m-1} (t^{\nu_i} - s^{\nu_i}) 
 = \frac{p_\nu}{s-t}
\,.
\]

Finally, by Lemma~\ref{lem:bigintersection} and 
Proposition~\ref{prop:commutationidentity},
we have
\begin{align*}
   F_{\sigma_{[n]}} 
   &= F_{\sigma_{[n]},\emptyset}
                    + \sum_{a=1}^n F_{\sigma_{[n]},\{a\}}
\\
   &= \Phi(p_\nu,q_{[n]})
                    + 
     \sum_{a=1}^n \Phi\Big(\frac{p_\nu}{s-t}, q_{a} q_{[n]}\Big)
\\
   &= \Phi\Big(\frac{p_\nu}{s-t},q'_{[n]}\Big)
   + \sum_{a=1}^n \Phi\Big(\frac{p_\nu}{s-t},q_{a} q_{[n]}\Big)
\\
   &= \Phi\Big(\frac{p_\nu}{s-t},q'_{[n]} + \sum_{a=1}^n q_{a} q_{[n]}\Big)
\,.
\end{align*}
The result now follows, because
$q'_{[n]} + \sum_{a=1}^n q_{a} q_{[n]} = 0$.
\end{proof}

\begin{proof}[Proof of Theorem~\ref{thm:main}]
By Lemma~\ref{lem:allzero}, the only $\sigma_Z \in \SP_n$ that
produces a non-zero summand on the right hand side 
of Lemma~\ref{lem:expansion} is the pair
$Z = \emptyset$, $\sigma = \idSn$, which yields
$\du^{n-|Z|} F_{\sigma_Z} \sigma \du^{n-|Z|} = \du^{2n}$.
\end{proof}

%%%%%%%%%%%%%%%%%%%%%%%%%%%%%%%%%%%%%%%%%%%%%%%%%%%%%%%%%
%
%
\section{Commutativity}
\label{sec:commute}

In this section, we give a bijective proof of the fact that
the operators $\beta_{k,l}^\pm$ all commute.  The proof is
is essentially identical for all sign combinations, so 
for ease of notation, we'll focus on 
\[
   \beta_{k,l}^+ \beta_{k',l'}^- = \beta_{k',l'}^- \beta_{k',l'}^+  
\,.
\]
We will treat $z_1, \dots, z_n$ as formal indeterminates.
Working formally, it suffices to prove the above identity 
in the case where $l' = n-k'$.  This is enough because 
if we know that 
$\beta_{k,l}^+ \beta_{k',n-k'}^- = \beta_{k',n-k'}^- \beta_{k,l}^+$,
then substituting $z_i \mapsto z_i+t$ we have
$\beta_{k,l}^+(t) \beta_{k',n-k'}^-(t) 
= \beta_{k',n-k'}^-(t) \beta_{k,l}^+(t)$ for all $t \in \CC$, 
from which one can easily deduce the other commutation relations.

Let $B_{k,l}$ denote the set of pairs $(\sigma_X,Y)$ where
$\sigma_X \in \SP_n$ is a supported permutation
and $Y \subseteq [n] \setminus X$, with
$|X| = k$, $|Y| = l$.
Then
\[
    \beta_{k,l}^+ = \sum_{(\sigma_X, Y) \in B_{k,l}} \sigma z_Y
\qquad
\qquad
    \beta_{k',l'}^- = \sum_{(\sigma'_{X'}, Y') \in B_{k',l'}} 
    \sgn(\sigma') \sigma' z_{Y'}
\,.
\]
and $\beta_{k,l}^+ \beta_{k',n-k'}^- = \beta_{k',n-k'}^- \beta_{k,l}^+$,
is the statement
\begin{equation}
\label{eqn:commute}
 \sum_{\substack{(\sigma_X,Y\,;\, \sigma'_{X'},Y') \\
\in ~ B_{k,l}\times B_{k',n-k'} }}
   \sgn(\sigma') \, \sigma \sigma'\, z_Yz_{Y'}
 = 
 \sum_{
\substack{
(\bar \sigma'_{\bar X'}, \bar Y' \,;\, \bar\sigma_{\bar X},\bar Y) \\
\in ~B_{k',n-k'} \times B_{k,l}}}
   \sgn(\bar \sigma') \, \bar \sigma' \bar \sigma\, z_{\bar Y}z_{\bar Y'}
\,.
\end{equation}

Define a preorder $\preceq$ on $\Sn$ as follows.  For $\pi, \tau \in \Sn$,
we'll say $\pi \preceq \tau$ if every fixed point of $\tau$ is a fixed
point of $\pi$.  The following two lemmas are straightforward.

\begin{lemma}
\label{lem:breakcycles}
Let $Z \subseteq [n]$. 
For every $\tau \in \Sn$ there exists permutation 
$\pi \in \symgp_{[n] \setminus Z}$ such that $\pi \preceq \tau$ and
every cycle of
$\pi \tau$ contains at most one element of $[n] \setminus Z$.
\end{lemma}

\begin{lemma}
\label{lem:reflect}
Let $Z \subseteq [n]$. 
Let $\hat \tau \in \Sn$ be a permutation such that every cycle
contains at most one element $[n] \setminus Z$.
Then there exists an involution
$\xi \in \symgp_{Z}$ such that 
$\xi \preceq \hat \tau$
and
$\hat \tau = \xi \hat\tau^{-1} \xi$.
\end{lemma}

For every pair $\tau \in \Sn$, $Z \subseteq [n]$, 
choose a permutation $\pi = \pi_{\tau,Z} \in \Sn$ as in 
Lemma~\ref{lem:breakcycles}
and let $\hat \tau = \pi_{\tau,Z} \tau$;
then choose an involution $\xi = \xi_{\tau,Z} \in \Sn$ 
as in Lemma~\ref{lem:reflect} and
let $\hat \xi_{\tau,Z} = \pi_{\tau,Z} \xi_{\tau,Z}$.
Note that $\xi_{\tau,Z}$ commutes with $\pi_{\tau,Z}$,
since $\pi_{\tau,Z} \in \symgp_{[n] \setminus Z}$ 
and $\xi_{\tau,Z} \in \symgp_{Z}$,
Note also that 
$\hat \xi_{\tau,Z} \preceq \tau$.

\begin{proposition}
\label{prop:commutebijection}
Consider the map
\begin{align*}
    \rho : B_{k,l} \times B_{k',n-k'} &\to B_{k',n-k'} \times B_{k,l}
\\
    \big(\sigma_X,Y \,;\, \sigma'_{X'},Y'\big) &\mapsto
    \big(\bar\sigma'_{\bar X'}, \bar Y'\,;\, \bar \sigma_{\bar X},\bar Y\big) 
\,,
\end{align*}
defined by
\begin{equation}
\label{eqn:thebijection}
{\begin{gathered}
   \bar X' = \hat \xi(X')\,,\quad
   \bar Y' = \hat \xi(Y')\,,
\\
   \bar X = \hat \xi(X)\,,\quad
   \bar Y = \hat \xi(Y)\,, 
\end{gathered}}\qquad \qquad
{\begin{aligned}
   \bar \sigma' &= 
   \hat \xi^{-1} (\sigma')^{-1} 
   \hat \xi^{-1}\,,
\\
   \bar \sigma &= 
   \hat \xi \sigma^{-1} \hat \xi^{-1}
\,,
\end{aligned}}
\end{equation}
where $\hat \xi = \hat \xi_{\sigma \sigma', Y \cup Y'}$.
Then $\rho$ is a bijection.

Furthermore, $\rho$ is weight preserving, in the sense 
that
$\sigma \sigma' = \bar \sigma' \bar \sigma$, 
$z_Y z_{Y'} = z_{\bar Y} z_{\bar Y'}$,
$\sgn(\sigma) = \sgn(\bar \sigma)$ and
$\sgn(\sigma') = \sgn(\bar \sigma')$.
\end{proposition}

\begin{proof}
Given $(\sigma_{X},Y\,;\, \sigma'_{X'}, Y') \in B_{k,l} \times B_{k',n-k'}$, 
write $Z = Y \cup Y'$, $\tau = \sigma \sigma'$,
$\pi = \pi_{\tau,Z}$, $\xi = \xi_{\tau,Z}$, 
$\hat \tau = \pi \tau$, and  $\hat \xi = \pi \xi$.

We begin by verifying that
$(\bar \sigma'_{\bar X'}, \bar Y'\,;\, 
\bar \sigma_{\bar X},\bar Y) \in B_{k',n-k'} \times B_{k,l}$.
The only part of this claim that's not clear is the assertion 
$\bar \sigma'_{\bar X'} \in \SP_n$.  To see this, rewrite
the formula for $\bar \sigma'$ as
$\bar \sigma' = 
   \hat \xi (\sigma' \pi^2)^{-1} \hat \xi^{-1}$,
and note that $\sigma', \pi \in \symgp_{X'}$.  (Remark: Here is where
we need the assumption $l' = n-k'$; this implies
$X' = [n] \setminus Y' \supseteq [n] \setminus Z$, whence 
$\pi \in \symgp_{X'}$.)

We check that $\rho$ is weight preserving.  First of all,
\[
   \bar \sigma' \bar \sigma 
=
   \hat \xi^{-1} (\sigma')^{-1} \sigma^{-1} \hat \xi^{-1}
=
   \pi^{-1} \xi(\sigma')^{-1} 
\sigma^{-1} \pi^{-1} \xi
=
   \pi^{-1} \xi
\hat \tau^{-1}
\xi
=
   \pi^{-1} \hat\tau
=
\tau =  \sigma \sigma'.
\]
Next, since $\pi \in \symgp_{[n]\setminus Z}$, and 
$\xi \in \symgp_{Z}$,
$Z = Y\cup Y'$ is an invariant subset for
both $\pi$ and $\xi$ and hence 
it is invariant for $\hat \xi$.  
Therefore,
$\bar Y \cup \bar Y' =  \hat \xi (Y \cup Y') = Y \cup Y'$.
Also, since $\hat \xi \preceq \tau$, we have
$\bar Y \cap \bar Y' =  \hat \xi (Y \cap Y') = Y \cap Y'$.
Together, these imply that $z_Y z_{Y'} = z_{\bar Y} z_{\bar Y'}$.
The fact that
$\sgn(\sigma) = \sgn(\bar \sigma)$ and 
$\sgn(\sigma') = \sgn(\bar \sigma')$ is clear from~\eqref{eqn:thebijection}.

Finally, we check that $\rho$ 
is a bijection.  Since the
domain and codomain have the same cardinality, it suffices
to prove that $\rho$ is injective.  Suppose that
$\rho(\sigma_{\!1\;X_1},Y_1\,;\, \sigma'_{\!1\;X'_1}, Y'_1) = 
\rho (\sigma_{\!2\;X_2},Y_2\,;\, \sigma'_{\!2\;X'_2}, X'_2) =
(\bar \sigma'_{\bar X'}, \bar Y'\,;\,
\bar \sigma_{\bar X},\bar Y)$.
Let $Z_i = Y'_i \cup Y_i$, 
and $\tau_i = \sigma_i \sigma'_i$, for $i=1,2$.
By the preceding remarks, 
$Z_1 = \bar Y \cup \bar Y' = Z_2$,
and $\tau_1 = \bar \sigma'\bar\sigma =\tau_2$.
Therefore $\hat\xi_{\tau_1,Z_1} = \hat\xi_{\tau_2,Z_2}$.  
From~\eqref{eqn:thebijection}, it follows that 
$(\sigma_{\!1\;X_1},Y_1\,;\, \sigma'_{\!1\;X'_1}, Y'_1) = 
(\sigma_{\!2\;X_2},Y_2\,;\, \sigma'_{\!2\;X'_2}, X'_2)$.
\end{proof}

We now give the proof of Theorem~\ref{thm:commute}, with one small caveat:  
the final case in the proof uses Theorem~\ref{thm:polyfunctions}, 
which is proved in Section~\ref{sec:duality}.  This does not lead
to any circularity, since the final case is not used by any 
of the arguments in Section~\ref{sec:schubert} or~\ref{sec:duality}.
The argument below establishes the commutativity of 
$\bethe(z_1, \dots, z_n)$, and the
equality of the different algebras in the case where
$(z_1, \dots, z_n)$ is a general point of $\CC^n$. 
The final case, where $(z_1, \dots, z_n) \in \CC^n$ is arbitrary, is where
we need Theorem~\ref{thm:polyfunctions}.

We note that the commutativity of $\bethe(z_1, \dots, z_n)$ is enough to 
infer Theorem~\ref{thm:invwrsols} from Theorem~\ref{thm:main}:
the equality of the three algebras is not needed for this argument.
Therefore, in the remaining sections of the paper, we will freely use
Theorem~\ref{thm:invwrsols}.

\begin{proof}[Proof of Theorem~\ref{thm:commute}]
The bijection $\rho$ in Proposition~\ref{prop:commutebijection}
corresponds terms on the left hand side of~\eqref{eqn:commute} 
with terms on the right hand side, which proves commutativity.
This shows that $\bethe^\pm(z_1, \dots, z_n)$ and 
$\bethe(z_1, \dots, z_n)$ are commutative subalgebras of $\CC[\Sn]$.

If $(z_1, \dots, z_n) \in \CC^n$ is general, then
$\bethe^\pm(z_1, \dots, z_n)$ are both maximal commutative
subalgebras of $\CC[\Sn]$.
This follows from the fact 
that for $(z_1, \dots, z_n)$ general, 
$\bethe^{\pm}(z_1, \dots, z_n)$
is a deformation of the Gelfand--Tsetlin subalgebra of $\CC[\Sn]$
\cite[Proposition 2.5]{MTV-Sn}, which is a maximal commutative
subalgebra (see~\cite{OV}).
We deduce that
$\bethe^-(z_1, \dots, z_n) = \bethe(z_1, \dots, z_n)
= \bethe^+(z_1, \dots, z_n)$
for $(z_1, \dots, z_n)$ general.

Proving this equality of algebras for arbitrary $(z_1, \dots, z_n)$ 
is a bit more involved.
We need to show that for all $k,l$, there exists a polynomial 
function which expresses $\beta^+_{k,l}$ in terms of
the operators $\beta^-_{k',l'}$, and vice-versa. This is the content
of Theorem~\ref{thm:polyfunctions}.
\end{proof}

%%%%%%%%%%%%%%%%%%%%%%%%%%%%%%%%%%%%%%%%%%%%%%%%%%%%%%%%%
%
%
\section{Schubert cells}
\label{sec:schubert}

Let $V = \langle f_1, \dots, f_n \rangle$ be an $n$-dimensional subspace 
of $\CC[u]$.  Write $d_i = \deg(f_i)$.
We may assume that our basis for $V$ is chosen such that
$d_1 > d_2 > \dots > d_n$.  Let $\lambda_i = d_i-n+i$.  Then
$\lambda = (\lambda_1, \lambda_2 , \dots, \lambda_n)$
is a partition.
(Note that here, some of the ``parts'' $\lambda_i$ may be zero.)
We say $V$ has \defn{Schubert type} $\lambda$, and the numbers
$d_1, \dots d_n$ are called the \defn{exponents of $V$ at infinity}.
The space of all $V$ of Schubert type $\lambda$ is called
a \defn{Schubert cell}, and is denoted $\scell$.
Note that $|\lambda| = \deg(\Wr_V)$.

The fundamental differential operator $D_V$ encodes the 
Schubert type $V$, as follows.  If $g(u) \in \CC(u)$ is
a non-zero rational function, we say that $c \in \CC^\times$ is the
\defn{leading coefficient} of $g$ if $c^{-1}g$ is monic; if $g = 0$,
the leading coefficient of $g$ is $0$.  Let $\indicial_k(D_V)$ 
denote the leading coefficient of $(-1)^{n-k}\Dcoeff{D_V}{k}$,
and let $\indicial(D_V) = (\indicial_0(D_V), \dots, \indicial_n(D_V))$.

\begin{proposition}
\label{prop:indicial}
Let $V, V'$ be $n$-dimensional subspaces of $\CC[u]$.
$V$ and $V'$ have the same exponents at infinity if and only if
$\indicial(D_V) = \indicial(D_{V'})$.
\end{proposition}

\begin{proof}
The exponents of $V$ at infinity are the roots of the \emph{indicial equation}
\[
    \sum_{k=0}^n (-1)^{n-k}\, \indicial_k(D_V) \cdot x (x-1) \dotsm (x-k+1) = 0
\,,
\]
(see e.g. \cite[\S 7.21]{Ince}).
\end{proof}

\begin{theorem}
\label{thm:schubert}
Let $\lambda \vdash n$, and let $E \subset M^\lambda$ 
be an eigenspace of the Bethe algebra, of type $\lambda$.
Then $V_E = \ker(\calD^-_E)$ is in the Schubert cell $\scell$.
\end{theorem}

\begin{proof}
For $g(u) \in \CC(u)$, we can write
\begin{equation}
\label{eqn:indicial}
   \Wr(\tfrac{u^{d_1}}{d_1!}, \dots, \tfrac{u^{d_n}}{d_n!},g)
   = 
    %\frac{\dim M^\lambda}{n!}
    %\left(
     \sum_{k=0}^{n-1} (-1)^{n-k} c^\lambda_k u^k g^{(k)}(u) 
    %\right)
\,,
\end{equation}
for some sequence of rational numbers $(c^\lambda_0, \dots, c^\lambda_n)$.
Up to a scalar multiple, $(c^\lambda_0, \dots, c^\lambda_n)$ is 
equal to $\indicial(D_{V_0})$,
where $V_0 =  \langle u^{d_1}, \dots, u^{d_n} \rangle \in \scell$.
By Proposition~\ref{prop:indicial}, 
for any $n$-dimensional subspace $V \subset \CC[u]$, we have
$V \in \scell$ if and only if 
$\indicial(D_V) = r(c^\lambda_0, \dots, c^\lambda_n)$ for some constant $r$.

Taking derivatives of both sides of~\eqref{eqn:indicial}, 
and using the fact 
\[
  \Wr(f_1, \dots, f_m)' 
  = \sum_{i=1}^m \Wr(f_1, \dots, f_{i-1}, f_i', f_{i+1}, \dots, f_m)
\,,
\]
we obtain the following recurrence for the numbers $c_k^\lambda$:
\[
     c^\lambda_k = 
   \begin{cases} 
   1  &\quad \text{if $k = 0$, $\lambda = 1^n$} \\
   0  &\quad \text{if $k = 0$, $\lambda \neq 1^n$} \\
   \sum_{\mu \lessdot \lambda} \frac{1}{k} c^\mu_{k-1} 
      &\quad \text{if $k \geq 1$.}
   \end{cases}
\]
In the last case, the sum is taken over all partitions $\mu \vdash n-1$ 
such that $\mu_i \leq \lambda_i$ for all $i$.

The elements $\beta^-_{n-k,0} \in \bethe^-(z_1, \dots, z_n)$ 
do not depend 
on $z_1, \dots, z_n$, and are in the centre of $\CC[\Sn]$.
Hence $\beta^-_{n-k,0}$ acts as a scalar $b^\lambda_k$ on $M^\lambda$.  
Considering the trace, we find that
\[
 \dim M^\lambda \cdot 
  b^\lambda_k 
= 
    \sum_{\substack{X \subseteq [n] \\|X| = n-k}}
       \chi^\lambda(\alpha_X^-)
        = \frac{n!}{k!}
            \langle \schur_\lambda ,\schur_{1^{n-k}}\schur_1^k \rangle
\,.
\]
Here $\chi^\lambda : \CC[\Sn] \to \CC$ denotes the character of 
$M^\lambda$, and 
$\schur_\lambda$ is the Schur symmetric function;
the second equality above uses the Frobenius characteristic map.
It is well-known that
$\langle \schur_\lambda, \phi \schur_1 \rangle = 
\sum_{\mu \lessdot \lambda} \langle \schur_\mu , \phi \rangle$
for any symmetric function $\phi$; hence the 
numbers $\frac{\dim M^\lambda}{n!} b^\lambda_k$ satisfy the same 
recurrence as the numbers $c^\lambda_k$, and we conclude that
$(b^\lambda_0, \dots, b^\lambda_n) = \frac{n!}{\dim M^\lambda} 
(c^\lambda_0, \dots, c^\lambda_n)$.
From the definition of $\calD^-_E$, we have that
$(b^\lambda_0, \dots, b^\lambda_n) = \indicial(\calD^-_E)$,
and therefore $V_E \in \scell$.
\end{proof}

We can now see that $\beta_{0,0}^-, \dots, \beta_{n,0}^-$
generate the centre of $\CC[\Sn]$.

\begin{proof}[Proof of Theorem~\ref{thm:centre}]
Suppose $\gamma_1, \dots, \gamma_m \in Z(\CC[\Sn])$, where
$\gamma_i$ acts as the scalar $\gamma^\lambda_i$ on $M^\lambda$.
The elements $\gamma_1, \dots, \gamma_m$ generate $Z(\CC[\Sn])$ 
if and only if the tuples $(\gamma^\lambda_1, \dots, \gamma^\lambda_m)$
are distinct for distinct partitions $\lambda \vdash n$.
The proof of Theorem~\ref{thm:schubert} establishes this for
the elements $\beta^-_{0,0}, \dots, \beta^-_{n,0} \in Z(\CC[\Sn])$.
\end{proof}

\begin{remark}
\label{rmk:bijectivecorrespondence}
If $z_1, \dots, z_n$ are generic, then there are exactly $\dim M^\lambda$
distinct solutions to the inverse Wronskian problem in the Schubert cell
$\scell$, and there are exactly $\dim M^\lambda$ eigenspaces of the
Bethe algebra of type $\lambda$.  The first statement is a computation
in the Schubert calculus (see e.g. \cite[\S 2.2]{LP}); 
the second statement follows from the fact that $\bethe^-(z_1, \dots, z_n)$ 
is a deformation of the Gelfand--Tsetlin 
algebra \cite[Proposition 2.5]{MTV-Sn}.  This numerical coincidence 
explains why every
$n$-dimensional solution to the inverse Wronskian problem is of
the form $V_E = \ker(\calD^-_E)$.
\end{remark}

\begin{remark}
None of the theorems discussed in this section are new.
Proposition~\ref{prop:indicial} is from the classical
theory of Fuchsian differential equations.
Theorem~\ref{thm:schubert} is implicitly part of the 
content of \cite[Theorem 4.3]{MTV-Sn}; the proof above is based on
the same main idea, but avoids using Schur--Weyl duality.
Theorem~\ref{thm:centre} is \cite[Proposition 2.1]{MTV-Sn}, and
the authors' assertion that this follows from \cite[Proposition 3.5]{MTV-Sn}
is essentially the proof given above.
\end{remark}

%%%%%%%%%%%%%%%%%%%%%%%%%%%%%%%%%%%%%%%%%%%%%%%%%%%%%%%%%
%
%
\section{Duality}
\label{sec:duality}

Let $\lambda = (\lambda_1, \dots, \lambda_n)$ be a partition,
and let 
$d_i = \lambda_i+n-i$.
Let $e_1 < e_2 < e_3 < \dotsb$ denote the non-negative integers distinct
from $d_1, \dots, d_n$.
For $V \in \scell$, define the \defn{canonical basis} of $V$ 
to be the unique basis 
$(f_1, \dots, f_n)$ of the form
\[
   f_i(u) = \frac{u^{d_i}}{d_i!} + 
\sum_{j=1}^{\lambda_i}
   (-1)^{1+n-i-j+e_j} 
    v_{ij}
   \frac{u^{e_j}}{e_j!}
  \,.
\]
The coefficients $(v_{ij})_{j \leq \lambda_i}$ of the canonical
basis polynomials are called the \defn{canonical coordinates} of $V$.

Let $\CC_{2n-1}[u] = \ker(\du^{2n})$ denote the $2n$-dimensional
vector space of
polynomials of degree at most $2n-1$, and
let $\Gr(n, \CC_{2n-1}[u])$ denote the Grassmannian variety of
$n$-dimensional linear subspaces of $\CC_{2n-1}[u]$.  As a set,
$\Gr(n, \CC_{2n-1}[u])$ is the union of all Schubert cells
$\scell$ for which $\lambda_1 \leq n$.

Now assume that $\lambda_1 \leq n$, and
let $\lambda^* = (\lambda^*_1, \dots, \lambda^*_n)$ 
denote the conjugate partition of $\lambda$.
If $V \in \scell$ has Schubert type $\lambda$ and canonical coordinates 
$(v_{ij})_{j \leq \lambda_i}$, then there is another $n$-dimensional
subspace $V^* \in \dualscell$ 
with Schubert type $\lambda^*$ and canonical
coordinates $(v^*_{ij})_{j \leq \lambda^*_i}$, such that 
$v^*_{ij} = v_{ji}$ for all $i,j$.  $V^*$ is called the 
\defn{Grassmann dual} of $V$.  

\begin{proposition}
\label{prop:samewronskian}
The map $V \mapsto V^*$ defines
an automorphism of the variety $\Gr(n,\CC_{2n-1}[u])$.
Moreover, for every $V \in \Gr(n,\CC_{2n-1}[u])$, we have
$\Wr_V = \Wr_{V^*}$.
\end{proposition}

\begin{proof}
See \cite[Remark 2.5]{LP}.
\end{proof}

\begin{example}
Let $V_{E_1}$ and $V_{E_2}$ 
be the solutions to the inverse Wronskian problem for $w(u) = u^3-3u$,
from Example~\ref{ex:21}:  
\begin{align*}
V_{E_1} &= \ker (\calD^-_{E_1}) = \langle u^4+4u^3, u^2-2u,1 \rangle
\\
V_{E_2} &= \ker (\calD^-_{E_2}) = \langle u^4-4u^3, u^2+2u,1 \rangle
\,.
\end{align*}
$V_{E_1}$ has canonical coordinates $(v_{11}, v_{12}, v_{21}) = (0,1,-1)$
and $V_{E_2}$ has canonical coordinates $(v_{11}, v_{12}, v_{21}) = (0,-1,1)$,
so these spaces satisfy $V_{E_1}^* = V_{E_2}$.
\end{example}

Let $\star : \Dalg[\Sn] \to \Dalg[\Sn]$ denote the algebra 
automorphism which acts trivially on $\Dalg$, and by
$\star \sigma = \sgn(\sigma)\sigma$ on $\CC[\Sn]$.  Hence,
\[
   \star\,\Big(\sum_{\sigma \in \Sn} \Psi_\sigma \sigma\Big)
   =
   \sum_{\sigma \in \Sn} \sgn(\sigma) \Psi_\sigma \sigma
\,.
\]
In particular, note that $\star \beta^+_{k,l} = \beta^-_{k,l}$
and $\star \beta^-_{k,l} = \beta^+_{k,l}$.  
Applying $\star$ to Theorem~\ref{thm:main}, we obtain:
\begin{corollary}
   $(\star \calD^+_n)(\star\calD^-_n) = \du^{2n}.$
\end{corollary}
Therefore, if $E \subset M^\lambda$ is
an eigenspace of $\bethe^+(z_1, \dots, z_n)$, then
$\ker(\star\calD^-_{E})$ is a solution to the inverse Wronskian problem
for $w(u)$.
This can also been seen from the fact that
$\star$ corresponds to tensoring with sign representation $\SgnRep$.
If $E\subset M^\lambda$ is an eigenspace of $\bethe^+(z_1, \dots, z_n)$, then
$E \otimes \SgnRep \subset M^\lambda \otimes \SgnRep = M^{\lambda^*}$
is manifestly an eigenspace of $\bethe^-(z_1, \dots, z_n)$.
We have $\beta^+_{k,l,E} = \beta^-_{k,l,E \otimes \SgnRep}$, so 
$\star \calD^-_{E} = \calD^-_{E \otimes \SgnRep}$.

This observation becomes more interesting in light of the
fact that $\bethe^+(z_1, \dots, z_n)$ commutes with
$\bethe^-(z_1, \dots, z_n)$.
If we take $E$ to be an eigenspace of
$\bethe(z_1, \dots, z_n)$,
then we have two solutions, $V_E = \ker (\calD^-_{E})$
and $V_{E \otimes \SgnRep} = \ker (\star \calD^-_{E})$, 
to the inverse Wronskian
problem, both associated to $E$.  These two solutions are related
by Grassmann duality.

\begin{theorem}
\label{thm:dual}
If $E$ is an eigenspace of $\bethe(z_1,\dots, z_n)$, then
$V_{E \otimes \SgnRep} = V_E^*$.
\end{theorem}

\begin{proof}
Proposition~\ref{prop:samewronskian} implies that
there exists an eigenspace $E^* \subset M \otimes \SgnRep$, such
that $V_E^* = V_{E^*}$.  We must show that $E^* = E \otimes \SgnRep$, for
$(z_1, \dots, z_n)$ general.  By continuity, this implies
the result for all $(z_1, \dots, z_n)$.  Note
that the relationship between $E^*$ and $E$ is completely determined 
by what happens at any general point $(z_1, \dots, z_n)$, by parallel
transport. So it is enough to prove this for $(z_1, \dots, z_n)$ 
belonging to some Zariski dense open subset of $\CC^n$. 

Consider
the degeneration of $\bethe(z_1, \dots, z_n)$ to the Gelfand-Tsetlin
algebra \cite[Proposition 2.5]{MTV-Sn}, 
which can be obtained by substituting $z_i \to t^iz_i$, and 
letting $t \to \infty$.
For $t$ large, this degeneration process allows us to
assign standard Young tableaux
$T_{V_E}$, $T_{V_E^*}$, $T_E$, $T_{E^*}$, $T_{E \otimes \SgnRep}$, 
to each of
$V_E$, $V_E^*$, $E$, $E^*$, $E \otimes \SgnRep$
(see \cite{OV, Pur-Gr, Pur-ribbon, Sp, W}).  
In the case
of $V \subset \CC[u]$, the tableau $T_V$ is defined in terms the asymptotics
of the coordinates of $V$; in the case of $E \subset M^\lambda$, the
limit is an eigenspace
of the Gelfand--Tsetlin algebra, which
naturally has an associated tableau.  
Each tableau uniquely
identifies the subspace of $\CC[u]$ or eigenspace of 
$\bethe(z_1, \dots, z_n)$ in question.

Furthermore, these tableaux are related.
Using the definition of $T_V$ from \cite[\S 2.1]{Pur-ribbon},
Theorem~\ref{thm:schubert} implies that
$T_{V_E} = T_E$ for any eigenspace $E$ of the Bethe algebra 
(see also \cite{W}).
It follows from Proposition~\ref{prop:samewronskian}
that $T_{V_E^*} = T_{V_E}^*$.
Finally, $T_{E \otimes \SgnRep} = T_E^*$ 
is a basic property of the Gelfand--Tsetlin algebra 
(see \cite{OV}).
Here if $T$ is a standard Young tableau, $T^*$ denotes the conjugate
tableau, obtained by reflecting $T$ along the main diagonal. 
Putting this all together, we have 
\[ 
  T_{E \otimes \SgnRep} = T_E^* = T_{V_E}^* = T_{V_E^*} %= T_{V_{E^*}} 
  = T_{E^*}
\,.
\]
Hence $E^* = E \otimes \SgnRep$ as required.
\end{proof}

We now use duality to finish the proof of Theorem~\ref{thm:commute},
showing that $\bethe^+(z_1, \dots, z_n) = \bethe^-(z_1, \dots, z_m)
= \bethe(z_1, \dots, z_n)$
for all $(z_1, \dots, z_n) \in \CC^n$.

\begin{lemma}
\label{lem:polysolution}
Let $V \in \scell$, with canonical coordinates
$(v_{ij})_{j \leq \lambda_i}$.  For $j \leq \lambda_i$, let
$s_{ij}$ be the coefficient of $u^{e_j}$
in $\Dcoeff{\Wr_V D_V u^{d_i}}{0}$.
Then $s_{ij}$ is given by a polynomial in the canonical 
coordinates with $\QQ$-coefficients, which is of the form
\[
      s_{ij} = c_{ij} v_{ij} + r_{ij}
\,,
\]
where $c_{ij} \in \QQ$ is a non-zero constant, and $r_{ij}$ is a polynomial
involving only the coordinates
   $\{v_{i'j'} \mid d_{i'}-e_{j'} < d_i - e_j\}$.
\end{lemma}

\begin{proof}
Let $(f_1, \dots, f_n)$ be the canonical basis for $V$.  
Up to an irrelevant non-zero scalar, $s_{ij}$
is equal to the coefficient of $u^{e_j}$ in
$\Wr(f_1, \dots, f_n, u^{d_i})$.
This is a polynomial
in the canonical coordinates with $\QQ$-coefficients. 
Rewriting the Wronskian as
\[
\Wr(f_1,\dots,  f_{i-1}, f_i-\tfrac{u^{d_i}}{d_i!}, f_{i+1}, \dots, , f_n, 
  u^{d_i})
\,.
\]
we see that $s_{ij}$ is a linear function of $f_i-\frac{u^{d_i}}{d_i!}$,
so each term in $s_{ij}$ must contain exactly one $v_{ij'}$ for
some $j' \leq \lambda_i$.

Now, think of $v_{ij}$ as an indeterminate of degree $d_i - e_j$; hence
$f_i$ is a homogeneous polynomial of degree $d_i$. 
Then $s_{ij}$ is homogeneous of degree
$d_i - e_j$, which means it can only involve
indeterminates of degree $d_i - e_j$ or less.  This, together with the
preceding remarks shows that $s_{ij} = c_{ij} v_{ij} + r_{ij}$, where
$r_{ij}$ only involves indeterminates of 
degree less than $d_i - e_j$.
Finally, 
\[
 c_{ij} = (-1)^{1+n-i-j+e_j} u^{-e_j}
\Wr(\tfrac{u^{d_1}}{d_1!},\dots,  \tfrac{u^{d_{i-1}}}{d_{i-1}!}, 
   \tfrac{u^{e_j}}{e_j!}, \tfrac{u^{d_{i+1}}}{d_{i+1}!}, \dots, 
\tfrac{u^{d_n}}{d_n!}, u^{d_i})
\,,
\]
which is non-zero, since the exponents $d_1, \dots, d_n, e_j$ are
distinct.
\end{proof}

\begin{theorem}
\label{thm:polyfunctions}
For all $k,l \leq n$ there exist polynomials with $\QQ$-coefficients
which express the operators $\beta^{+}_{k,l}$ as a function the operators
$\beta^{-}_{k',l'}$, and vice-versa.
\end{theorem}

\begin{proof}
First note that if we have a formula for $\beta^+_{k,l}$ in terms
of the $\beta^-_{k',l'}$, then applying $\star$ to both sides gives
a formula for $\beta^{-}_{k,l}$ in terms of the $\beta^+_{k',l'}$,
so the ``vice-versa'' statement will be automatic.

Let $\beta^\pm_{k,l,\lambda} \in \End(M^\lambda)$ and 
$\calD^\pm_\lambda \in \Dalg \otimes \End(M^\lambda)$ denote
the restrictions of operators $\beta^\pm_{k,l}$ and $\calD^\pm_n$ to 
$M^\lambda$.
Let $P_\lambda \in Z(\CC[\Sn])$ denote the central idempotent which acts
as the identity on $M^\lambda$, and as zero on $M^{\lambda'}$, 
$\lambda' \neq \lambda$.
By Theorem~\ref{thm:centre}, $P_\lambda$
is given by some polynomial with $\QQ$-coefficients
in $\beta_{0,0}^-, \dots, \beta_{n,0}^-$.
We will prove that there exist polynomials $Q_{k,l, \lambda}$
with $\QQ$-coefficients
which express $\beta^+_{k,l,\lambda}$ in terms of $\beta^-_{k',l',\lambda}$
for each $\lambda \vdash n$.  This is sufficient, as
$\sum_{\lambda \vdash n} Q_{k,l,\lambda}P_\lambda$ will then give
a polynomial expression for
$\beta^+_{k,l}$ in terms of $\beta^-_{k',l'}$.

Let $\Delta^\lambda : \scell \to \dualscell$ denote the
map $V \mapsto V^*$.
Let $\calY_n \subset \Dalg$ be the vector space of differential
operators $\Psi$ of order at most $n$, such that 
$\Dcoeff{\Psi}{i}$ is a polynomial of degree at most $i$.
Let $\Omega^\lambda : \scell \to \calY_n$ be the map defined
by $\Omega^\lambda(V) = \Wr_V D_V$.  
Both $\Delta^\lambda$ and $\Omega^\lambda$ are defined by polynomials
with $\QQ$-coefficients.

Lemma~\ref{lem:polysolution} shows that
$\Omega^\lambda : \scell \to \calY_n$ 
has an 
left-inverse $\Upsilon^\lambda: \calY_n \to \scell$, defined
by polynomials with $\QQ$-coefficients: given $\Psi = \Wr_V D_V$,
we can solve for the canonical coordinates $(v_{ij})_{j \leq \lambda_i}$ 
of $V$, 
recursively, in increasing order of $d_i - e_j$.

Now consider the composition 
\[
\Theta^\lambda = 
  \Omega^{\lambda^*} \circ \Delta^\lambda \circ \Upsilon^\lambda
\,,
\]
which is a polynomial map from $\calY_n$ to itself.
By definition, for all $V \in \scell$,
$\Theta^\lambda(\Wr_V D_V) = \Wr_{V^*}D_{V^*}$.
Thus, if
$E \subset M^\lambda$ is an eigenspace of $\bethe(z_1, \dots, z_n)$, then
by Theorem~\ref{thm:dual},
\[
\Theta^\lambda(w \calD^-_{E}) = w \calD^-_{E \otimes \SgnRep} 
= \star (w \calD^-_{E})
\,.
\]
Now assume $(z_1, \dots, z_n) \in \CC^n$ is general, so that
$M^\lambda$ is a direct sum of eigenspaces of $\bethe(z_1, \dots, z_n)$.
Then we have
$\Theta^\lambda(w \calD^-_\lambda) = \star (w \calD^-_\lambda)$.
Finally, since this is a polynomial identity which holds for 
$(z_1, \dots, z_n)$ general, 
it holds for all $(z_1, \dots, z_n) \in \CC^n$.
Therefore the required polynomials $Q_{k,l,\lambda}$ are just the 
coordinates of the map $\Theta^\lambda$.
\end{proof}

\begin{remark}
The maps $\Omega^\lambda$ and $\Upsilon^\lambda$ in the proof of 
Theorem~\ref{thm:polyfunctions} are essentially the isomorphism
described in \cite[Theorem 4.3(iv)]{MTV-Sn} and its inverse.
The proof of Lemma~\ref{lem:polysolution} is based on
\cite[Lemma 4.5]{MTV-transverse}.
\end{remark}

%%%%%%%%%%%%%%%%%%%%%%%%%%%%%%%%%%%%%%%%%%%%%%%%%%%%%%%%%
%
%
\section{Discussion}
\label{sec:conclusion}

Theorem~\ref{thm:invwrsols} and the results of 
Sections~\ref{sec:schubert}
and~\ref{sec:duality} mainly focus on the differential 
operators $\calD^-_n$ and $\calD^-_E$.
It not obvious what the corresponding story is for $\calD^+_n$.

Let $E$ be an eigenspace of $\bethe(z_1, \dots, z_n)$, and consider
the subspace $V^+_E = \ker (\calD^+_E) \subset \CC(u)$.  
We would like to know which subspaces of $\CC(u)$ are of this form.
As already noted in the introduction, Theorem~\ref{thm:main} tells us 
that $V^+_E$ is $n$-dimensional, with $\Wr_{V^+_E} = \frac{1}{w}$.
We now state a slightly stronger necessary condition.

\begin{proposition}
If $V^+_E = \ker (\calD^+_E)$, where $E$ is an eigenspace of
the Bethe algebra, then $w V^+_E = \{wg \mid g \in V^+_E\}$ 
is an $n$-dimensional
vector space of polynomials, which is an $n$-dimensional 
solution to the inverse Wronskian problem for $w^{n-1}$.
Furthermore $w V^+_E \in \scell$ for some partition
$\lambda = (\lambda_1, \dots, \lambda_n)$
with $\lambda_1 \leq n$.
\end{proposition}

\begin{proof}
It also follows from Theorem~\ref{thm:main} that $\ker (\calD^+_E)$
is contained in the image of $\calD^-_E$ restricted to $\ker (\du^{2n})$.  
In particular, $\ker (\calD^+_E)$ has a basis of the form 
$(\frac{f_1}{w}, \dots, \frac{f_n}{w})$,
where $f_1, \dots, f_n$
are polynomials of degree at most $2n$.  This shows that
$w V^+_E \subset \CC[u]$, and its Schubert type $\lambda$ satisfies 
$\lambda_1 \leq n$.
Finally we have the identity 
\[
\Wr(gf_1, \dots, g f_n) = g^n \cdot \Wr(f_1, \dots, f_n)
\]
for any $g(u) \in \CC(u)$. Hence the fact that 
$\Wr_{V^+_E} = \frac{1}{w}$ implies that 
$\Wr_{w V^+_E} = w^n \cdot \frac{1}{w}$.
\end{proof}

The converse is false.  If $V$ is an $n$-dimensional solution to the 
inverse Wronskian problem for $w^{n-1}$, with appropriate conditions
on the Schubert type,
it is not necessarily true that $V = w V^+_E$ for some eigenspace
of the Bethe algebra $E$.  We have a pretty good guess what the
right sufficient condition is.

\begin{conjecture}
\label{conj:D+}
Suppose $z_1, \dots, z_n$ are distinct.
Let $V \subset \CC[u]$ be an $n$-dimensional vector space of
Schubert type $\lambda$, such that $\Wr_V = w^{n-1}$.
Then $V = wV^+_E$ for some eigenspace $E$ of
$\bethe(z_1, \dots, z_n)$, 
if and only if $V$ belongs to the Schubert intersection 
\[
   \scell \cap X_{1^{n-1}}(z_1) \cap \dots \cap X_{1^{n-1}}(z_n)
\,.
\]
\end{conjecture}
Here $X_\mu(z)$ for $z \in \CC$ is a Schubert variety inside
$\Gr(n, \CC_{2n-1}[u])$; we are following the notation and conventions
of \cite[\S 2.1]{LP}.
It should be possible to prove this
by applying the machinery of \cite{MTV-transverse} to the 
$\gl_n(\CC[t])$-representation
$(\exterior{n-1} \CC^n)^{\otimes n}$, and using Schur--Weyl duality.
The author has verified that this works up to $n=3$, 
but a complete proof is beyond the intended scope of this paper.

If Conjecture~\ref{conj:D+} is correct, it still does not fully
characterize $\ker (\calD^+_E)$.  
A more complete answer would describe the
precise the relationship between $\ker (\calD^-_E)$ and 
$\ker (\calD^+_E)$, analogously to the way Theorem~\ref{thm:dual}
describes the relationship between 
$\ker (\calD^-_E)$ and $\ker (\star \calD^-_E)$.
One might hope that understanding this relationship could lead to a more
conceptual proof of Theorem~\ref{thm:main}.

A natural question is whether there is a more general form
of Theorem~\ref{thm:main}, for example, an identity inside
the full Bethe algebra in $\uea$, rather than 
just inside the Bethe subalgebra
of $\CC[\Sn]$.  One problem with this notion is that
in formulations of Theorem~\ref{thm:invwrsols} using the full
Bethe algebra, there is no uniform upper bound on the degrees of the 
polynomials involved; instead 
one has different bounds for different representations of $\gl_m$.
By contrast, working in $\CC[\Sn]$, we can say for any eigenspace 
$E \subset M^\lambda$,
$\ker (\calD^-_E)$ only involves polynomials of degree at most $2n$,
independent of $\lambda$.
V. Tarasov has pointed out that the results of \cite{TU}
provide a sort of analogue.  In particular Section~6 therein 
describes a factorization of a fixed differential operator, of the 
same form as our Theorem~\ref{thm:main}.  These results are
based on the Bethe ansatz in the context of 
$(\gl_k,\gl_m)$ duality on the exterior algebra 
$\extalg(\CC^k \otimes \CC^m)$.  The authors show that 
the image of
the Bethe algebras for $\gl_k$ and $\gl_m$ coincide on the representation
$\extalg(\CC^k \otimes \CC^m)$.   A similar result holds
for the Lie superalgebras $\gl_k$ and $\gl_{m|n}$ \cite{HM}.
This suggests that operators 
$\calD^-_n$ and $\calD^+_n$ should be regarded as
elements of a quotient of this common image, but coming from the two 
different factors of $\gl_n$ acting 
$\extalg(\CC^n \otimes \CC^n)$.
It would be great to see this worked out explicitly.

A related question is whether Theorem~\ref{thm:main} has an analogue 
for the XXX model.  In \cite{MTV-Sn}, Mukhin, Tarasov and Varchenko
define a Bethe subalgebra of $\CC[\Sn]$ of XXX type, 
which is a $1$-parameter deformation of $\bethe(z_1, \dots, z_n)$.
A paper of Uvarov \cite{Uvarov} generalizes the results of
\cite{TU} in this direction, and it would be valuable to have
versions of these results formulated concretely inside $\CC[\Sn]$.

Finally, it would be nice to have a more explicit formula for
the polynomials which express $\beta^+_{k,l}$ in terms 
in terms of $\beta^-_{k',l'}$, or for the map $\Theta^\lambda$
defined in the proof of Theorem~\ref{thm:polyfunctions}.  Since the elements
$\beta^+_{k,l}$ are (or are at least related to) ``coefficients'' 
of $\calD^+_n$, this may shed some light on the aforementioned problem 
of describing 
the relationship between $\ker (\calD^+_E)$ and $\ker(\calD^-_E)$.

%%%%%%%%%%%%%%%%%%%%%%%%%%%%%%%%%%%%%%%%%%%%%%%%%%%%%%%%%%%%%%
%%%%%%%%%%%%%%%%%%%%%%%%%%%%%%%%%%%%%%%%%%%%%%%%%%%%%%%%%%%%%%

\bigskip

\footnotesize%
\noindent
\textsc{K. Purbhoo, Combinatorics and Optimization Department, 
       University of Waterloo,  
       Waterloo, ON, N2L 3G1, Canada.} \texttt{kpurbhoo@uwaterloo.ca}.


\begin{thebibliography}{00}

\bibitem{EG-pole} A. Eremenko and A. Gabrielov,
\emph{Pole placement static output feedback for generic linear systems,}
SIAM J. Control Optim. \textbf{41} (2002), no. 1, 303--312 (electronic).

\bibitem{EG-elementary}
A. Eremenko and A. Gabrielov,
\emph{An Elementary Proof of the B. and M. Shapiro Conjecture for 
Rational Functions.} 
In: P. Br\"and\'en, M. Passare, M. Putinar (eds),
Notions of Positivity and the Geometry of Polynomials,  (2011) 
Trends in Mathematics, Springer, Basel. 

\bibitem{HM}
C. Huang and E. Mukhin,
\emph{The duality of $\gl_{m|n}$ and $\gl_{k}$ Gaudin models},
preprint, arXiv:1904.02753.

\bibitem{Ince}
E.\,L. Ince,
\emph{Ordinary Differential Equations},
Dover, New York, 1956.

\bibitem{LP} J. Levinson and K. Purbhoo,
\emph{A topological proof of the Shapiro--Shapiro conjecture},
Invent. Math. \textbf{226} (2021), 521--578.

\bibitem{MTV-reality}
E. Mukhin, V. Tarasov and A. Varchenko,
\emph{The B. and M. Shapiro conjecture in real algebraic geometry
and the Bethe Ansatz}, Ann. Math., \textbf{170} (2009), no. 2, 863--881.

\bibitem{MTV-transverse}
E. Mukhin, V. Tarasov and A. Varchenko,
\emph{Schubert calculus and representations of general linear group},
J. Amer. Math. Soc.,  \textbf{22}  (2009),  no. 4, 909--940.

\bibitem{MTV-Sn}
E. Mukhin, V. Tarasov and A. Varchenko,
\emph{Bethe subalgebras of the group algebra of the symmetric group},
Transformation Groups \textbf{18} (2013), 767--801.

\bibitem{OV}
A. Okounkov and A. Vershik, 
\emph{A new approach to representation theory of symmetric
groups. II}, J. Math. Sci. \textbf{131} (2005), no. 2, 5471--5494.

\bibitem{Pur-Gr}
K. Purbhoo,
\emph{Jeu de taquin and a monodromy problem for Wronksians
of polynomials}, Adv. Math.,  \textbf{224} (2010) no. 3, 827--862.


\bibitem{Pur-ribbon}
K. Purbhoo,
\emph{Wronskians, cyclic group actions, and ribbon tableaux},
Trans. Amer. Math. Soc., \textbf{365} (2013), 1977--2030.

\bibitem{sage}
W.\,A. Stein et~al., \emph{{S}age {M}athematics {S}oftware
({V}ersion 9.5)}, The Sage Development Team, 2022, {\tt
http://www.sagemath.org}.

\bibitem{Sp}
D. Speyer,
\emph{Schubert problems with respect to osculating flags
of stable rational curves},
Alg. Geom., \textbf{1} (2014) no. 1, 14--45.


\bibitem{Sot-F}
F. Sottile,
\emph{Frontiers of reality in Schubert calculus},
Bull. Amer. Math. Soc., \textbf{47}  (2010),  no. 1, 31--71.

\bibitem{TU}
V. Tarasov and F. Uvarov,
\emph{Duality for Bethe algebras acting on polynomials 
in anticommuting variables}, preprint, arXiv:1907.02117.


\bibitem{Uvarov}
F. Uvarov,
\emph{Difference operators and duality for 
trigonometric Gaudin and dynamical Hamiltonians}
SIGMA \textbf{18} (2022) 081.

\bibitem{W}
N. White,
\emph{Labelling Schubert intersections in the Grassmanian},
preprint, arXiv:1912.10090.

\end{thebibliography}
\end{document}